\documentclass{article}
\usepackage[utf8]{inputenc}
\usepackage[margin=1in]{geometry}
\pdfoutput=1
\usepackage{multirow}
\usepackage{amsmath}
\usepackage{float}
\usepackage{amsmath}
\usepackage{amsthm}

\newtheorem{theorem}{Theorem}
\newtheorem*{remark}{Remark}
\usepackage{amssymb}

\newtheorem{definition}{Definition}[section]

\usepackage{graphicx}

\date{}

\begin{document}

\title{Global Optimisation in Hilbert Spaces using the Survival of the Fittest Algorithm.}

\author{Andrew Yu. Morozov$^{1,2,3}$ \\
Oleg Kuzenkov$^{3}$ \\
Simran K. Sandhu$^{1\ast}$}

\maketitle

\noindent{} 1. Department of Mathematics, University of Leicester, LE1 7RH, UK;

\noindent{} 2. Institute of Ecology and Evolution, Russian Academy of Sciences, Moscow, Russia

\noindent{} 3. Lobachevsky State University of Nizhny Novgorod, Nizhny Novgorod, Russia

\noindent{} $\ast$ Corresponding author; e-mail: sks55@le.ac.uk.

\begin{abstract}
Global optimisation problems in high-dimensional and infinite dimensional spaces arise in various real-world applications such as engineering, economics and finance, geophysics, biology, machine learning, optimal control, etc. Among stochastic approaches to global optimisation, biology-inspired methods are currently very popular in the literature. Bio-inspired approaches imitate natural ecological and evolutionary processes and are reported to be efficient in a large number of practical study cases. On the other hand, many of bio-inspired methods can possess some vital drawbacks. For example, due to their semi-empirical nature, convergence to the globally optimal solution cannot always be guaranteed. Another major obstacle is that the existing methods often struggle with the high dimensionality of space (approximating the underlying functional space), showing a slow convergence. It is often difficult to adjust the dimensionality of the space of parameters in the corresponding computer code for a practical realisation of the optimisation method. Here, we present a bio-inspired global stochastic optimisation method, applicable in Hilbert function spaces. The proposed method is an evolutionary algorithm inspired by Darwin's' famous idea of the survival of the fittest and is, therefore, referred to as the `Survival of the Fittest Algorithm' (SoFA). Mathematically, the convergence of SoFA is a consequence of a fundamental property of localisation of probabilistic measure in a Hilbert space, and we rigorously prove the convergence of the introduced algorithm for a generic class of functionals. The approach is simple in terms of practical coding. As an insightful, real-world problem, we apply our method to find the globally optimal trajectory for the daily vertical migration of zooplankton in the ocean and lakes, the phenomenon considered to be the largest synchronised movement of biomass on Earth. We maximise fitness in a function space derived from a von-Foerster stage-structured population model with biologically realistic parameters. We show that for problems of fitness maximisation in high-dimensional spaces, SoFA provides better performance as compared to some other stochastic global optimisation algorithms. We highlight the links between the new optimisation algorithm and natural selection process in ecosystems occurring within a population via gradual exclusion of competitive con-specific strains.
\end{abstract}

\section{Introduction}

Global optimisation problems in high-dimensional finite as well as infinite-dimensional Hilbert spaces, arise in a large number of applications within different research areas such as mechanics, optics, geophysics, economics, finance, machine learning, biology, etc. Stochastic approaches are currently widely implemented in global optimisation. In particular, the tremendous progress in molecular biology, genetics, microbiology, and foraging ecology made within the recent 30-40 years has led to the emergence of a series of new stochastic optimisation techniques known as bio-inspired methods. Such algorithms of optimisation imitate ecological processes  (e.g. functioning of ant or bees colonies, populations of bacteria, swarms of krill, etc.) as well as those in genetics and biochemistry (e.g. replication and mutation of DNA) \cite{deb2001multi, passino2012bacterial,mavrovouniotis2017survey, rai2013bio}. The objective functional in bio-inspired optimisation is biological fitness, which, is understood in a broad sense. Among bio-inspired optimisation methods, evolutionary algorithms - such as genetic algorithms and differential evolution - are considered to be the key ones. These methods try to mimic long-term biological evolution, including the processes of mutation, recombination, selection and reproduction \cite{deb2001multi}. Genetic algorithms encode information about an individual via a binary representation, whereas differential evolution considers a non-binary representation \cite{back1996evolutionary, storn1997differential}.  The other types of bio-inspired algorithms include swarm intelligence, which is based off interacting searching agents sharing information \cite{di2015optimal, mavrovouniotis2017survey}, and bacterial foraging, which imitate the social foraging behaviour of \textit{Escherichia coli} or other microorganisms (\cite{passino2012bacterial}) or even humans (\cite{volchenkov2013exploration}).

Despite the existence of dozens of promising optimisation techniques imitating biological processes, there are still crucial challenges of their usage \cite{mavrovouniotis2017survey}. The central question is about the eventual convergence of a particular method towards the optimal solution for an increasing number of iterations. Indeed, despite an algorithm having a good record of previously successful implementations, it cannot be guaranteed, without rigorous analytical proof, that the considered method would locate the optimal solution for some new class of problems.
The other shortcoming of the existing bio-inspired techniques is that most of them are not well-designed to find the optimal solution in function spaces characterised by an infinitely large number of dimensions (in practice, of course, there is still an upper limit for the maximal dimensionality). For example, swarm intelligence algorithms are best fit to choose the optimal path out of the existing combinations. On the other hand, for problems with complex constraints on the parameters or/and functions, optimisation methods can get stuck in non-feasible domains (e.g. producing non-viable mutants in the course of evolution) which substantially reduces the speed of optimisation. Adjustment of evolutionary algorithms to deal with the unfeasible realisation of parameters can be achieved, in principle \cite{deb2001multi}; however, this would increase the complexity of the methods preventing their extensive applicability.

In this paper, we introduce a novel method of bio-inspired stochastic global optimisation in high-dimensional or even infinite-dimension Hilbert spaces, named Survival of the Fittest Algorithm (SoFA), which can successfully cope with the pre-mentioned difficulties. SoFA is close to the genetic algorithms of search, however unlike previously proposed methods, it takes advantage of the the fundamental principles of population ecology, in particular, the competitive exclusion principle \cite{odum1971fundamentals}. As the algorithm name suggests, that is the quantification of Darwin's' famous idea of the survival of the fittest \cite{darwin2009origin}. Mathematically, the method uses the recently developed approach of modelling biological evolution based on measure dynamics and measure localisation around the point(s) with the maximal fitness \cite{gorban2007selection,Kuzenkov_measure,Kuzenkov_N,kuzenkov2019towards}. A crucial advantage of SoFA is its convergence which holds for an arbitrary positive objective functional (fitness function) in a Hilbert space, and here we present the corresponding rigorous proof. The convergence rate of SoFA is easily adjustable by tuning the parameters describing mutation rates. Finally, the computational algorithm of the method is extremely straightforward in its usage and does not require specific knowledge of the subject, further improving its practical applicability. 


To demonstrate the great potential of the presented method in solving  optimisation problems in complex biological systems (this method was in fact inspired by such systems), we predict the globally optimal trajectory for the diel vertical migration (DVM) of zooplankton in the ocean. Note that although DVM of zooplankton is recognised to be the largest synchronised movement of biomass on the planet \cite{hays2003,kaiser2011}, this phenomenon is still not well-understood and mathematical modelling would play a role to bridge the current gaps in the area. In our model, each developmental stage of zooplankton is mathematically characterised as a continuous function describing the vertical position of an organism in the water column, so the problem requires the use of Hilbert spaces. The objective functional, defined by population fitness, is derived from the well-known von-Foerster stage-structured population model. We use biological relevant parameters describing DVM of a dominant zooplankton herbivorous species in the north-eastern Black Sea. We compare our method with some other stochastic global optimisation techniques including ESCH \cite{da2010,da2010Thesis}, CRS \cite{kaelo2006,price1983,price1977}, and MLSL \cite{kan1987I,kan1987II} which are well-known in the optimisation literature. We show that SoFA exhibits a higher performance (in terms of reduction of the average error with the number of iterations), in the case where the parametric space is highly dimensional. In particular, the proposed method deals well with biological constraints required all functionals (e.g. mortality, reproduction, maturation, etc.) in the model to be positive which is not the case of the considered ESCH, CRS and MLSL methods.

The paper is structured as follows. In Section 2, we introduce the basic algorithm of the method (Section 2.1) and explain how to apply the optimisation procedure to reveal the optimal trajectory of diel vertical migration (DVM) of zooplankton (Section 2.2). In the Results section, we rigorously prove the convergence of the global optimisation method (Section 3.1) and construct the optimal trajectories of DVM of zooplankton based on the stage-structured model (Section 3.2). We also compare the efficiency of the proposed method with the existing global optimisation methods (Section 3.3). The Discussion section addresses the connection between the new optimisation method and biological evolution processes. The Summary section concludes the study.

\section{Methodology}

\subsection{Description of the basic SoFA framework \label{sec:Algorithm}}

Consider a standard Hilbert space $l_2$ of functions $z$ where infinite sequences are square-summable, i.e.

\begin{equation}
l_2=\left\{z : z=(x_1, x_2, \ldots,x_k, \ldots), \left( \sum_{n=1}^{\infty} x_n^2 \right)^{1/2}<\infty\right\}.
\end{equation}

We take an element $c=(c_1, c_2, \ldots c_k, \ldots)$ with positive components and denote $R=\big( \sum_{n=1}^{\infty} c_n^2 \big)^{1/2}<\infty$. Based on the above element we introduce the following infinite dimensional cube (parallelepiped) $\Pi=\{z : |x_n-a_n|\le c_n/2, n=1,2,\ldots \}$. Suppose there is some continuous positive function (functional) $J(z)$ which is defined in $\Pi$. Assuming that $J(z)$ has an unique point of maximum (denoted by $z^*$), it can be located by implementing a novel optimisation method. The method is outlined in the following algorithm (SoFA). The algorithm will use the following auxiliary function determined by $f(r) = -\frac{r^2}{2R^2}$.

The formal description of the algorithm is the following.

\begin{enumerate}
  \item Randomly choose a point $z_1 = (x_{11}, 0, 0, . . .)$ assuming a uniform distribution in the one-dimensional projection $\Pi_1 = [- c_1/2 , c_1/2 ]$ of the cube  $\Pi$. The functional $J$ is evaluated at this point, i.e. $J(z_1)$.
  
\item The point $z_1$ is taken as the current reference point (denoted by $\bar{z}$) to find the  $z_2$. The second iteration point $z_2 = (x_{12}, x_{22}, 0, 0, . . .)$ is randomly chosen in the two-dimensional projection $\Pi_2 = [- c_1/2 , c_1/2 ]\times [- c_2/2 , c_2/2 ]$ of the cube $\Pi$. The probability density distribution of $z_2$ is given by
\begin{equation*}
    \frac{2^{f(\|z-\bar{z}\|)}}{\int_{\Pi_2} \, 2^{f(\|z-\bar{z}\|)} \,d\Pi_2}.
\end{equation*}

The functional is evaluated at the point $J(z_2)$.
\item The new reference point $\bar{z}$ is randomly chosen out of $z_i$ ($i=1,2$) with the probability defined by $J^2(z_i)/(J^2(z_1)+J^2(z_2))$. Using the updated reference point $\bar{z}$, the third iteration point $z_2 = (x_{13}, x_{23}, x_{33}, 0, 0, . . .)$ is obtained, which is random vector generated in the three-dimensional projection $\Pi_3 = [- c_1/2 , c_1/2 ]\times [- c_2/2 , c_2/2 ] \times [- c_3/2 , c_3/2 ]$ of the cube $\Pi$. The probability density distribution of $z_3$ is given by
\begin{equation*}
    \frac{3^{f(\|z-\bar{z}\|)}}{\int_{\Pi_3} \, 3^{f(\|z-\bar{z}\|)} \,d\Pi_3}.
\end{equation*}

The functional is evaluated at the point $J(z_3)$.

\item Assume $k$ steps of the method have already been completed and therefore one has $k$ points $z_i$, $i=1,\ldots ,k$, with known corresponding values of $J(z_1),\ldots ,J(z_k)$. Then randomly select some reference point $\bar{z}$ out of all previously found points $z_i$. The probability of selecting each of these available points is given by $\frac{J^k(z_i)}{J^k(z_1)+\ldots +J^k(z_k)}$, $i=1,\ldots ,k$.

\item Using the above reference point $\bar{z}$, select a new point $z_{k+1}$ which is a random vector generated in the (k+1)-dimensional projection $\Pi_{k+1} = [- c_1/2 , c_1/2 ]\times [- c_2/2 , c_2/2 ] \dotsm \times [- c_{k+1}/2 , c_{k+1}/2 ]$ of the cube $\Pi$. The probability density distribution of $z_{k+1}$ is given by
\begin{equation}
    \frac{(k+1)^{f(\|z-\bar{z}\|)}}{\int_{\Pi_{k+1}} \, (k+1)^{f(\|z-\bar{z}\|)} \,d\Pi_{k+1}}.
     \label{eq:probability}
\end{equation}

The functional is evaluated at the point $J(z_{k+1})$, then the above steps are repeated.

\item The method will terminate once the approximation satisfies some initially prescribed requirement. In the simplest case, this can be a maximum number of iterations, although this simple criterion cannot necessarily guarantee an appropriate approximation of the maximal point. Another criterion can be that the distribution of points around the optimal should approach to the delta function (e.g. the standard deviation of the distribution should be smaller than a certain threshold). 
\end{enumerate}

\begin{remark}
Note that, technically, one can use a simplified version of the optimisation algorithm defined above (see Appendix A for details). In particular, we approximate the multi-dimensional function (\ref{eq:probability}) in the probability density distribution for the point $z_{k+1}$ by a product of simpler one-dimensional functions each of which given by 
 $$\frac{A^j_{k+1}}{\epsilon_{k+1}+(x_{j,k+1}-\bar{x}_j)^2},$$
where $A^j_{k+1}$ is a normalising constant, $\epsilon_{k+1}$ is the parameter characterising the standard deviation of the distribution, $\bar{x}_j$ is the $j^{th}$ coordinate of the current reference point $\bar{z}$. By choosing a particular parameterisation of $\epsilon_{k+1}$, one can regulate the rate of convergence of the optimisation algorithm. Finally, in practice, one can add extra dimensions only after a certain number of iterations (i.e. not at every iteration as in the basic version of SoFA); the dimensions can be added in blocks (i.e. adding several dimensions at some iteration step). After achieving a high dimensionality, one can keep the same number of dimensions until the end of the optimisation procedure, this is because the negative effects of computational noise in a higher dimensional system can undermine the positive effects of adding extra coordinates when approximating the function space with a finite number of dimensions.

\end{remark}

\begin{remark}
The main idea of SoFA is combining the processes of selection and mutation at each step. Unlike  genetic algorithms and methods of differential evolution, SoFA does not mimic the phenomenon of crossover.

In SoFA, evolutionary selection within the population is mimicked when choosing the reference point $\bar{z}$ in each iteration step. We can explain the underlying biological rationale in a non-rigorous way by considering the following simple discrete population model. We suggest that we have $k$ competing subpopulations within the population.  The growth of the subpopulation $i$ ($i=1,..,k$) with density $y_i$ from the start of year $s$ to the end of the same year denoted by $s^+$ is described by
 $$ y_i(s^+)= y_i(s)(1+a_i),$$
where $a_i>0$ is the per capita population growth rate for the given subpopulation. For simplicity, it is assumed to be constant and the generation time equals one year. We consider that by the end of each year the population is harvested proportionally to the current densities of each subpopulation in a way the total population density is kept at some carrying capacity (we assume it to be unity without the loss of generality). In this case, in the beginning of year $s+1$ the population density $y_i$ will be given by 
 $$ y_i(s+1)= \frac{y_i(s^+)}{\sum_j y_j(s^+)}=\frac{y_i(s)(1+a_i)}{\sum_j y_j(s)(1+a_j)}.$$
We start with the initial densities of the subpopulations such that their sum is exactly the carrying capacity of the system. In this case, one can easily derive the density $y_i(k^+)$ at the end of year $k$   $$ y_i(k^+)= \frac{(1+a_i)^k}{\sum_j (1+a_j)^k}.$$
By denoting $J_i=1+a_i$, we can re-write the above ratio as 
  $$y_i(k^+)= \frac{J_i^k}{\sum_j J_j^k}.$$
One can see that with long times (large $k$), the density of the subpopulation $y_i$ having the maximal fitness $\max J_i$ will tend to unity, whereas the proportion of the others will tend to zero. This fact explains the idea of the method where at each iteration, the selected reference point is in the proposed ratio form.

The mutation in the new optimisation method is modelled by an integral term (\ref{eq:probability}) with a Gaussian kernel. The centre of this kernel corresponds to the parent strain with the life-history trait $\bar{z}$, which produces mutant offspring. Note the choice of the parent strain is based on its reproductive success determined by $\frac{J_i^k}{\sum_j J_j^k}$. 




\end{remark}




\subsection{Modelling optimal DVM of zooplankton}
We apply the above optimisation method to model the particularly thought-provoking ecological case study, which is the regular daily vertical migration of zooplankton. Diel vertical migration (DVM) of marine and freshwater zooplankton in the water column, is regarded to be the largest synchronised movement of biomass on Earth \cite{hays2003,kaiser2011}. It is also vital to develop our understanding of patterns of DVM as it heavily impacts the biochemical cycles in the ocean, playing a fundamental role in the carbon exchange between the deep and surface waters, the ocean's biological pump and thus, the climate \cite{ducklow2001,buesseler2007,bianchi2013,hansen2016}. The typical pattern of DVM consists of the zooplankton organisms ascending to the phytoplankton rich surface waters for feeding at night, then descending to deeper depths and remaining there during the day \cite{hays2003,ohman1994}. There are currently several explanations of what is the ultimate cause of this mass migration. The most widely accepted hypothesis is that the zooplankton performs DVM to avoid visual predation (mostly by planktivorous fish) by spending daylight hours in the deeper darker waters and migrating up at night when these visual predators cannot see them \cite{ohman1994,fortier2001,pearre2003,lampert1989}.


The phenomenon of DVM has been studied extensively, both empirically \cite{ohman1994,fortier2001,pearre2003} and theoretically through the use of several mathematical models \cite{clark2000,ringelberg2009,morozov2011}; however, there are still some important gaps in our knowledge of zooplankton DVM.  The major challenge in modelling zooplankton DVM is to properly define the criterion of optimality since the use of different criteria may result in distinct predictions \cite{Fiksen95, han2001control,liu2003diel, morozov2016towards}. Recently, there has been the proposal of a new rigorous approach to identify the evolutionary fitness in systems with inheritance \cite{morozov2016towards,kuzenkov2019towards}. The approach considers the long-term dynamics of competing subpopulations which are described by different inherited units. 
Evolutionary fitness is defined based on the comparative ranking order of the subpopulations characterised by different behavioural strategies or life-history traits. In other words, a subpopulation will outcompete all other subpopulations possessing a lower ranking, with this ranking order being defined by the fitness function.
The proposed idea of modelling biological evolution in \cite{morozov2016towards,kuzenkov2019towards} is similar to that of SoFA framework which makes it natural to apply of the new optimisation method to maximise evolutionary fitness, in particular in DVM of zooplankton.

Using the above theoretical approach, the expression for evolutionary fitness was determined for some age-structured population models \cite{kuzenkov2019towards,morozov2019}. The need for age-structured models is justified by the fact that different zooplankton developmental stages exhibit distinct migration patterns of DVM \cite{Fiksen95,morozov2019}. Using a von-Foerster-type equation \cite{cushing1998,botsford1994} with continuous age but  discrete stages Morozov et al. derived the following expression for evolutionary fitness \cite{morozov2019} 
\begin{equation}
    J(v)=\frac{\max_i{\mathbb{R}(\lambda_i(v))}}{R(v)},
    \label{eq:fitness1}
\end{equation}
where $v$ is a continuous vector function of time describing the daily vertical trajectory for the considered developmental stages, $\lambda_i$ is the eigenvalue of the appropriate characteristic equation (see \cite{kuzenkov2019towards,morozov2019} for details) and $\mathbb{R}$ denotes the real part of this eigenvalue. $R(v)$ is the functional describing intraspecific competition within the population. For simplicity we assume $R(v) \equiv 1$. Note that Morozov and co-authors \cite{morozov2019} explored the the optimal trajectories in of DVM only based on a piecewise linear approximation, moreover, the authors used local rather than global optimisation methods. As such, it would be important to reveal the smooth underlying DVM trajectories which are `the exact' solution of the global optimisation problem (\ref{eq:fitness1}). 

As in \cite{morozov2019}, we consider three different developmental stages of herbivorous zooplankton: young stages, juveniles and adults. In this case, the characteristic equation for the eigenvalue $\lambda$ is given by \cite{morozov2019}:
\begin{equation}
    \lambda=b\exp{\left(-a_Y\tau_1-a_J(\tau_J-\tau_Y)\right)}\left[\exp{\left(-\tau_J \lambda \right)}-\exp{\left(-\tau_A \lambda-a_2(\tau_A-\tau_J)\right)}\right]-a_A,
    \label{eq:fitness}
\end{equation}
where $i=Y,J,A$ correspond to young, juvenile and adult stages, respectively. Here, $a_i$ is the mortality rates which accounts for the natural losses a due to visual predators (varying throughout the day due to dependence on light intensity), additional morality at the boundaries of unfavourable zones and natural non-predatory based mortality. $\tau_i$ are the maturation times of each stage; $\tau_A$ denotes the maximal reproduction age of adults. $b(v)$ is the reproduction coefficient which quantifies the number of eggs produced by a female in any given day. The terms $b(v)$ and $\tau_i$ incorporate the energy gained from feeding on phytoplankton, loses due to basal metabolism and the metabolic cost spent on active movements in the water when feeding and moving upwards while ascending. 

Maximisation of $J$ defined by (\ref{eq:fitness}) will locate the optimal strategies for all three developmental stages $v=(v_Y,v_J,v_A)$. The parameterisation of the integrals for computing $a_i$, $\tau_i$ and $b$ are assumed to be the same as in \cite{morozov2019}, for brevity, we do not include here the corresponding expressions. The model parameters used in the integrals are defined to be the default values from Table 1 of the cited work. The main difference between this study and that of Morozov et al. \cite{morozov2019} is that here the zooplankton is considered to be feeding when at a shallow enough depth and with their vertical speed being less than than the minimum threshold given by $c_0=10m/h$. Whereas Morozov et al. assumed that zooplankton grazers only feed when staying at shallowest depths during DVM and therefore did not allow for feeding during any vertical movement.

To reconstruct the trajectory in the Hilbert function space, we consider the following Fourier expansion
\begin{equation}
    v_s(t)=v_{s,1}+\sum^{N}_{m=1}\left(v_{s,2m}\sin(2\pi tm)+v_{s,2m+1}\cos(2\pi tm)\right)
     \label{eq:fouier}
\end{equation}
for $s=Y,J,A$. The maximal order $n=2N+1$ ($N=0,1,2,...$) in the above Fourier expansion can be set as large as possible. We substitute $v_s(t)$ in the integrals for $a_i$, $\tau_i$ and $b$ in the equation for fitness. We combine the trajectories of the all considered developmental stages in a single vector $v$ of dimension $3n$ setting $v_Y=v_1,....,v_{n}$, $v_J=v_{n+1},...,v_{2n}$ and $v_A=v_{2n+1},...,v_{3n}$ to perform optimisation in the space of dimension $D=3n$.





To evaluate the effectiveness of our method, along with the other long-standing optimisation methods, we introduce the following error definition.
\begin{definition}{\textbf{Function Error}}\label{Def:Function_Error}
Let the true maximal Fourier coefficients to be defined by the vector $v^*$ of dimension $D=3n$, then we define the function error of the approximated trajectory with Fourier coefficients given by $v$ as
\begin{equation}
    Err=J(v^*)-J(v)
    \label{eq:Err} 
\end{equation}
\end{definition}



We can also define the probability of convergence of an optimisation algorithm in the following way.
\begin{definition}{\textbf{Probability of Convergence.}}\label{def:Probability_of_Convergence}
Let the true maximal point be $v^*$ and the approximation of the optimal solution be $v$. Then the probability of convergence of the fitness to the 'true' optimal value $J^*$ with error $\delta>0$ is defined as follows
\begin{equation}
    P_{\delta}=P(|J(v^*)-J(v)|<\delta)
     \label{eq:P_delta} 
\end{equation}
\end{definition}
By fixing the value of $\delta$, we evaluate the performance of optimisation by plotting the ratio of successful approximations (i.e. those below the tolerance $\delta$) for an increasing number of iterations. Note that, usually the exact value of $v^*$ is unknown, so in this study, we use an approximation to the optimal solution by running simulations for a large number of iterations.


Along with SoFA, we also implement three other well-known stochastic optimisation algorithms: ESCH (Evolutionary Strategy with Cauchy distribution) \cite{da2010,da2010Thesis}, CRS (Controlled Random Search with local mutation) \cite{kaelo2006,price1983,price1977}, and MLSL (Multi-Level Single-Linkage) \cite{kan1987I,kan1987II}. The comparison of the efficiency of optimisation techniques is made based on the above-introduced error and the probability of convergence. Note that the maximal eigenvalue in the implicit equation for fitness (\ref{eq:fitness}) is, numerically, found using the Levenberg-Marquardt algorithm. Finally, it is also important to mention that within some given domains of $v$, it is possible to have biologically irrelevant results with, for example, negative mortality or maturation rates. Here, we denote these points $v$ as unfeasible points. We will plot the percentage of such points to effectively compare the efficiency of different global optimisation methods used to reveal the optimal DVM of zooplankton. 

Finally, to compare the optimisation in an infinite-dimensional function space we also construct the optimal trajectory based on piecewise linear approximation consider in the study of Morozov et al. \cite{morozov2019}. The corresponding equations are given by 
\begin{equation}
    \bar{v_s}(t)=  \left\{
\begin{array}{ll}
      H_{i0} & 0\leq t<t_{i0} \\
      c_{i0}(t-t_{i0})+H_{i0} & t_{i0}\leq t<t_{i1} \\
      H_{i1} & t_{i1}\leq t<t_{i2} \\
      c_{i1}(t-t_{i2})+H_{i1} & t_{i2}\leq t<t_{i3} \\
      H_{i0} & t_{i3}\leq t\leq 1
      \label{eq:linear} 
\end{array} 
\right.
\end{equation}
for $s=Y,J,A$. Here, the shallowest and deepest depths are given by $H_{i0}$ and $H_{i1}$, respectively; the speeds of ascending and descending are $c_{i0}$ and $c_{i1}$, respectively; the times of the end of each movement phase is described by $t_{ij}$. In this case, the number of unknown parameters is 3 for each stage which makes the overall number of parameters to be $D=9$. This is because in piecewise linear setting the optimal trajectories are symmetrical with respect to $t=0.5$ (see \cite{morozov2019} for detail)

We choose the domains for each of the Fourier coefficients $v_{s,i}$ in (\ref{eq:fouier}) to be the smallest possible domain such that both the `true' optimal point $v^*_{s,i}$ and the initial starting points (determined by the Fourier expansions of the piecewise linear `true' optimal trajectory).

\section{Results}

\subsection{Proof of convergence of SoFA \label{sec:proof}}

Here we rigorously demonstrate the convergence of the Survival of the Fittest Algorithm (SoFA) in the cube $\Pi$ in Hilbert space introduced in Section \ref{sec:Algorithm}. The proof consists of the two following theorems. 

 \begin{theorem}
The sequence $(z_1, . . . , z_k, . . .)$ is everywhere dense with the probability of unity, i.e. for any $z\in\Pi$ the probability to have a point $z_i$ in its neighbourhood  $O_{\epsilon}(z)$ tends to unity for large $m$.
 \end{theorem}
 \begin{proof}
 As $f(r) = -\frac{r^2}{2R^2}$, it is easy to see that $f(0)=0$ and for any two elements $h$ and $g$ within the cube $\Pi$ we have $\left \| h-g \right \| < R$, thus $f(\left \| h-g \right \|)>f(R) = -1/2$ (we define the constant $R$ in Section \ref{sec:Algorithm}).
 
We take a positive number $\epsilon<1/2$ and arbitrary element $z=(x_1, \ldots ,x_n, \ldots )\in\Pi$, thus we can introduce the $\epsilon$ neighbourhood $O_{\epsilon}(z)$ of $z$. Since the series $ \sum_{n=1}^{\infty} c_n^2$ is convergent ($c_n$ are introduced in Section \ref{sec:Algorithm}), there should exist a number $N$ such that $\sum\limits_{n=N+1}^{\infty} c_n^2\leq \frac{\epsilon^2}{2}$. 
Consider a point $y=(y_1,\ldots,y_n,\ldots)\in\Pi$. In the case its components satisfy the condition
$|y_n-x_n|\leq\frac{\epsilon}{2^n}, n=1,\ldots,N,$ then the point $y$ should belong to $O_{\epsilon}(z)$. 
Indeed, one can see that $$\sum\limits_{n=1}^{\infty} (x_n-y_n)^2=\sum\limits_{n=1}^{N} (x_n-y_n)^2+\sum\limits_{n=N+1}^{\infty} (x_n-y_n)^2\leq\sum\limits_{n=1}^{N} (\frac{\epsilon}{2^n})^2+ \sum\limits_{n=N+1}^{\infty} (c_n)^2\leq\epsilon^2.$$
Therefore, the set $\omega=\{ y=(y_1,\ldots y_n,\ldots): \left |y_n-x_n\right | \leq\frac{\epsilon}{2^n}, n=1,\ldots,N \}$ is actually a subset of $O_\epsilon (z)$.

Now assume that $k\geq N$ iterations of the method have been completed. The probability $P_{k+1}(\omega)$ for the point $z_{k+1}$ to lend in $\omega$ is given by
$$P_{k+1}(\omega)= \frac{\int\limits_{\omega} (k+1)^{f(\left \| z-\overline{z} \right \|)}d\Pi_{k+1}}{\int\limits_{\Pi_{k+1}}(k+1)^{f(\left \| z-\overline{z} \right \|)}d\Pi_{k+1}}.$$

We estimate the lower bound of the above probability;

$$P_{k+1}(\omega)= \prod \limits_{n=1}^N \frac{\int\limits_{x_n-\frac{\epsilon}{2^n}}^{ x_n+\frac{\epsilon}{2^n}} (k+1)^{f(\left | y_n-\bar{x}_n\right |)} dy_n}{\int\limits_{-c_n/2}^{c_n/2} (k+1)^{f(\left | y_n- \bar{x}_n\right |)}dy_n}\geq(k+1)^{-\frac{1}{2}} \prod\limits_{n=1}^N\frac{2\epsilon}{c_n2^n }=(k+1)^{-\frac{1}{2}}\mu^*,$$
where $\mu^*=\prod\limits_{n=1}^N\frac{2\epsilon}{c_n2^n }$ is a constant which does not depend on the number of iteration; $\bar{x}_n$ are the coordinates of the reference point $\bar{z}$. The probability to not land in the set $\omega$ within $M$ consecutive iterations starting from $N+1$ to $N+M$ can be estimated as
$$P_M \leq (1-\mu^*(M+N)^{-\frac{1}{2} })^M.$$
Since we have
$$ 
\lim\limits_{M\to \infty} (1-\mu^*(M+N)^{-\frac{1}{2} })^M =0,
$$
then the probability of non-choosing a point from $\omega$ for $M$ iterations tends to zero with $M \to \infty$.

Therefore, the sequence of points $z_k$ is everywhere dense in $\Pi$ with the probability of unity. 
 \end{proof}

 \begin{theorem}
 
 Let a continuous positive function $J(z)$ defined in $\Pi$ has an unique point of maximum given by $z^*=(x_1^*,x_2^*,\ldots x_k^*, \ldots)$. Then for any $\epsilon>0$ the probability of choosing a point $z_k$ from the neighbourhood $O_{\epsilon}(z^*)$ tends to unity when $k$ becomes infinitely large. In other words, the optimisation algorithm converges.
  \end{theorem}
 \begin{proof}

As it was shown in the proof of the previous theorem, the set $\omega(z^*)=\{ y=(y_1,\ldots y_n,\ldots): \left | y_n-x^*_n\right | \leq\frac{\epsilon}{2^n}, n=1,\ldots,N\}$ is actually a subset of the neighbourhood $O_{\varepsilon} (z^*)$. We estimate the probability of landing of the point $z_k$ in $\omega(z^*)$.

We introduce the following definitions: $J_0=\sup\limits_{z\in \Pi \backslash\omega(z^*)}J(z)$;
$I_m$ is the set of indexes of the points $z_i; i= 1,\ldots,k$ which get into $\omega(z^*)$; $\overline{I_k}$ is the set of indexes of the points $z_i, i=1,\ldots,k$ which do not get into $\omega(z^*)$.

From Theorem 1, $z_k$ is an everywhere dense sequence (with the probability of unity) in $\Pi$. Since we assume that $J$ is a continuous function, there will always be a point $z_p\in\omega(z^*)$ such that we have $J(z_p)>J_0$.

 For $k>p$, the probability to choose the point $z_j$ as a reference point $\overline{z}$ with the index $j$ from the set $\overline{I_k}$, can be estimated as follows;

$$\frac{\sum \limits_{j \in \overline{I_k}}J^k(z_j) }{\sum \limits_{j=1}^{k}J^{k}(z_j) }\leq \frac{\sum \limits_{j \in \overline{I_k}}J_{0}^{k} J^{-k}(z_p) }{\sum \limits_{j=1}^{k}J^{k}(z_j) J^{-k}(z_p)}<\left (\frac{J_0}{J(z_p)}\right )^{k}\xrightarrow{k \to \infty} 0$$

Since the above probability tends to zero, then the probability of choosing the point $z_j$ as a reference point $\overline{z}$ with the index $j$ from the set $I_k$ tends to unity. 

For $k>N$ we estimate the probability to choose the point $z_{k+1}$ under the condition that the reference point we obtain $z_j$ with the index $j$ from the set $I_k$, i.e. $\overline{z}\in\omega(z^*)$.

$$P_{k+1}(\omega(z^*))= \prod \limits_{n=1}^N \frac{\int\limits_{x_n-\frac{\epsilon}{2^n}}^{ x_n+\frac{\epsilon}{2^n}} (k+1)^{f(\left | y_n-\overline{x_n}\right |)} dy_n}{\int\limits_{-c_n/2}^{c_n/2} (k+1)^{f(\left | y_n-\overline{x_n}\right |)}dy_n}\geq$$

$$\geq\Phi\Big(\frac{x^*_n+ \frac{\epsilon}{2^n}-\overline{x_n}}{R}\sqrt{\ln ( k+1)}\Big)-\Phi\Big(\frac{x^*_n- \frac{\epsilon}{2^n}-\overline{x_n}}{R}\sqrt{\ln ( k+1)}\Big),$$

where
$$\Phi(t)=\frac{2}{\sqrt{\pi}}\int\limits_0^t e^{-\tau^2}d\tau.$$

Since $x^*_n+ \frac{\epsilon}{2^n}-\overline{x_n}>0$, we have
$$\frac{x^*_n+ \frac{\epsilon}{2^n}-\overline{x_n}}{R}\sqrt{\ln ( k+1)}\xrightarrow{k \to \infty}+\infty, 
\quad \Phi(\frac{x^*_n+ \frac{\epsilon}{2^n}-\overline{x_n}}{R}\sqrt{\ln ( k+1)})\xrightarrow{k \to \infty} 1.$$

Since $x^*_n- \frac{\epsilon}{2^n}-\overline{x_n}<0$, we have
$$\frac{x^*_n- \frac{\epsilon}{2^n}-\overline{x_n}}{R}\sqrt{\ln (k+1)}\xrightarrow{k \to \infty}-\infty, \quad \Phi(\frac{x^*_n- \frac{\epsilon}{2^n}-\overline{x_n}}{R}\sqrt{\ln( k+1)})\xrightarrow{k \to \infty}0.$$
Therefore, 
$$P_{k+1}(\omega(z^*))\geq\Phi\Big(\frac{x^*_n+ \frac{\epsilon}{2^n}-\overline{x_n}}{R}\sqrt{\ln ( k+1)}\Big)-\Phi\Big(\frac{x^*_n- \frac{\epsilon}{2^n}-\overline{x_n}}{R}\sqrt{\ln ( k+1)}\Big)\xrightarrow{k \to \infty}1.$$
The above limit finalises the proof of convergence of the proposed optimisation method.
 \end{proof}

\subsection{Simulating optimal trajectories of DVM of zooplankton \label{sec:DVM}}

We apply SoFA to reveal in optimal trajectories of DVM of zooplankton using the functional $J$ given by (\ref{eq:fitness}). The infinite-dimensional space will be approximated by the $n=2N+1$ first terms in Fourier series (\ref{eq:fouier}). The number of the considered Fourier terms  will determine the accuracy of the approximation of the optimal DVM patterns. One key advantage of SoFA is that one can adjust the speed of convergence by appropriately selecting the rate of decrease in the function $\epsilon(k)$ (we implement the simplified version of SoFA, see the first Remark in Section 2.1). Here we consider the parameterisation given by $\epsilon(k)=k^{-\alpha(k)}$, where $\alpha(k)$ is an increasing function of $k$. For simplicity, we consider the linear dependence $\alpha(k)=a+bk$ (with $a,b>0$); in Section 4 we briefly discuss the role of parameterisation of  $\epsilon(k)$ on the efficiency of the method.
\begin{figure}[H]
    \centering
    \includegraphics[trim=0 0 45 0,clip,width=\textwidth]{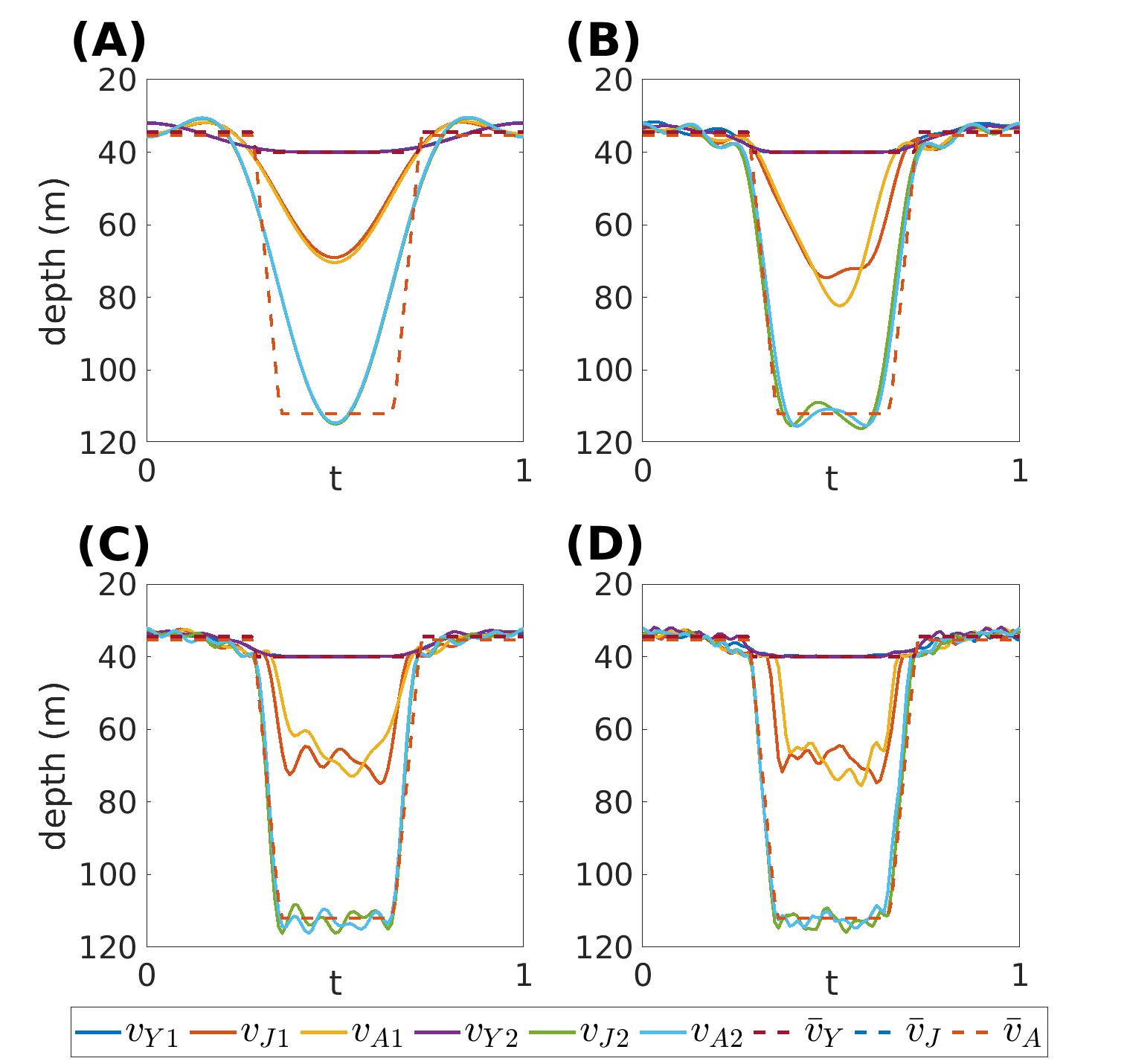}
    \caption{Optimal patterns of DVM of zooplankton obtained using SoFA for varying order of the Fourier series  (\ref{eq:fouier}) shown for each of the three developmental stages. (A) Trajectories are each defined by 5 Fourier terms with the globally maximum strategy being that with the overall shallower depth ($v_1$), unlike the corresponding piecewise linear optimal trajectories ($\bar{v}$). (B-D) Trajectories are each defined by 15, 27 and 55 Fourier terms, respectively, with the globally maximum strategy being at deeper depths ($v_2$), similar to the corresponding piecewise optimal trajectories ($\bar{v}$). Time is re-scaled such that t=0 is equivalent to midnight and t=0.5 is midday. The strategy for the first local maximum ($v_1=(v_{Y1},v_{J1},v_{A1})$) is indicated by the blue, orange and yellow curves (representing youths, juveniles and adults respectively), the strategy for the second local maximum ($v_2=(v_{Y2},v_{J2},v_{A2})$) is indicated by the purple, green and cyan curves. The dashed curves represent the optimal DVM for the case when trajectories are symmetric piecewise linear functions.}
    \label{fig:Low_Order}
\end{figure}

The optimal trajectories of DVM, constructed using SoFA, are presented in Fig.\ref{fig:Low_Order} shown for a progressively increasing number of Fourier terms in (\ref{eq:fouier}). The figure also shows the approximation of the trajectories using a piecewise linear function (\ref{eq:linear}), depicted by the dashed lines. Finally, we show other local maximum possible in the system to demonstrate the need for usage of global optimisation. Technically, to find a local maximum we implemented the local methods (realised in the MATLAB function \textit{fminsearch} based on the Nelder-Mead simplex algorithm) starting from different initial conditions.

 From Fig.\ref{fig:Low_Order} one can conclude that only the older developmental stages of zooplankton (denoted by $J$ and $A$) exhibit pronounced vertical migrations, whereas, the youngest stage ($Y$) remains in the surface waters all day. This pattern is observed for both the piecewise linear and smooth approximation of fitness. Fig.\ref{fig:Low_Order} shows that the two local maxima of the fitness function $J$ give two very different trajectories; the first ($v_1=(v_{Y1},v_{J1},v_{A1})$) with a shallower depth of around 70m, unlike the piecewise linear optimal DVM patterns ($\bar{v}$). The second ($v_2=(v_{Y2},v_{J2},v_{A2})$) goes to much deeper depths of approximately 110m, as was observed with the piecewise linear trajectories ($\bar{v}$). In the case with only $n=5$ Fourier terms used for each of the three age groups, the global fitness is achieved by the shallower patterns of DVM, dissimilar to the optimal piecewise linear DVM patterns. By increasing the order to $n=15$ terms, the globally maximal fitness is attained, by the trajectories similar to the optimal piecewise linear DVM patterns, i.e. those that travel to a deeper depth. Therefore, these two results demonstrate that the addition of more Fourier terms (meaning an increase in order) results in a drastic change in optimal trajectories, causing a switch in global maximum between the two local maxima. Following this, we also investigated very high orders of Fourier expansions such as $n=27$ or even $n=55$ terms to describe each trajectory. Fig.\ref{fig:Low_Order}(C, D) shows that for both these higher Fourier orders there exists two local maxima, furthermore, the global maximum is achieved by the trajectories that migrate to deeper depths.


We investigated the influence of the total number of Fourier terms, per trajectory, $n$ used (i.e. the system dimensionality $D=3n$) on the absolute value of evolutionary fitness $J$, with the results shown in Fig.\ref{fig:Compare_All_Orders}. We found that $J$ initially increases with $n$ up to around 27 terms (with $D=81$) and then remains almost constant. Note that, the small oscillations in $J$ for a large number of Fourier terms, can be partially explained by effects of noise on the system. In the same figure, we also indicate (dashed yellow line) the fitness corresponding to piecewise linear approximation (\ref{eq:linear}). The figure shows that fitness is much lower with a piecewise linear trajectory (yellow dashed line) when compared to that of the global maximum based on smooth curves (solid red curve). An important conclusion can be drawn from Fig.\ref{fig:Compare_All_Orders}, that using too many terms in the approximation of fitness seems to be computationally inefficient due to effects of noise as well as a long computational time needed to operate a high-dimensional system. For this reason, we consider Fourier approximations containing either 15 or 27 terms for the comparison with the other optimisation methods.



\begin{figure}[H]
    \centering
    \includegraphics[trim=25 0 200 0,clip,width=\textwidth]{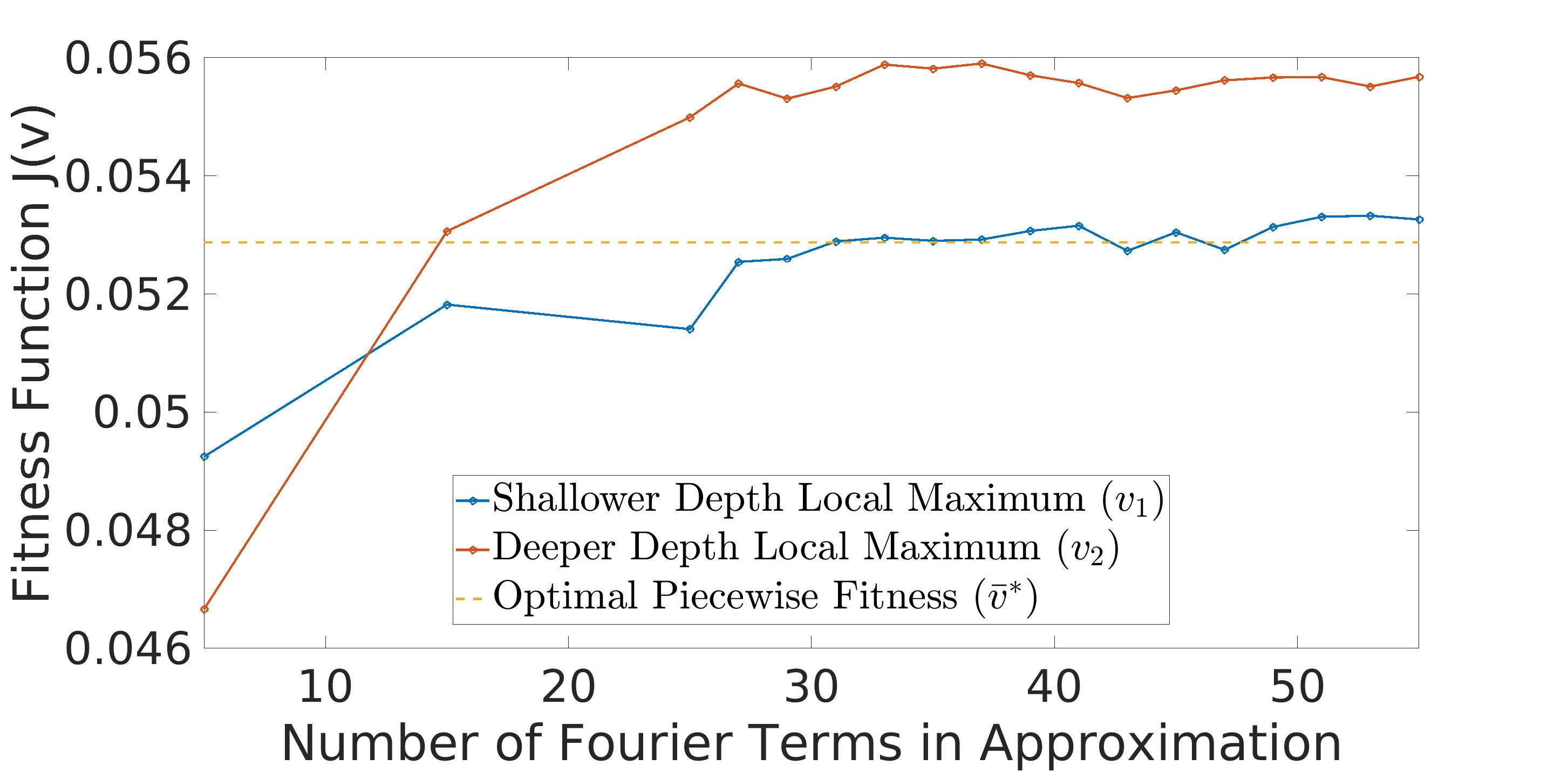}
    \caption{Dependence of fitness $J$ of different strategies of DVM of zooplankton on the number of Fourier terms used in (\ref{eq:fouier}). Fitness of the optimal piecewise linear trajectories $J(\bar{v})$, is indicated by the dashed yellow curve. The two local maxima, $J(v_1)$ and $J(v_2)$, shown by the blue and red curves, respectively, with the global maximum $J(v^*)$ given by the maximum of these two curves.}
    \label{fig:Compare_All_Orders}
\end{figure}

It is important to compare the efficiency of SoFA with some existing bio-inspired global algorithms such as ESCH, CRS and MLSL. For each of these methods, we run 200 realisations, with each of these realisations run for $2\times 10^{5}$ iterations. Fig.\ref{fig:Compare_15} presents the comparison of the efficiency of the fore-mentioned global optimisation algorithms for the Fourier series approximating the DVM trajectory. As a comparison metric, we use the error function $Err$ (shown in upper left panel of Fig.\ref{fig:Compare_15}) and the probability of convergence $P_{\delta}$ for different values of the accuracy $\delta$ (shown in the bottom panels of Fig.\ref{fig:Compare_15}). To understand possible slowness of a method, we also plot the percentage iterations which produce non-viable biological quantities (e.g. negative values of reproduction, maturation time, etc.). The percentage of unfeasible trajectories, displayed in the upper right panel of Fig.\ref{fig:Compare_15}, the greater this percentage, the larger the number of `wasted' iterations that give an approximation of $v$ that are biologically irrelevant. 

For the method introduced in this paper, we found the percentage of unfeasible trajectories is zero. From Fig.\ref{fig:Compare_15}, one can see that SoFA shows the best performance both in terms of the functional error and the probability of convergence. The panels (A) and (B) only differ in the overall number of the Fourier terms used. The figure shows that the novel SoFA optimisation is effective at converging towards the maximum as demonstrated through the steady decrease in the function error. Implementing SoFA results in a rapid convergence towards the global maximum, reaching an error less than $2\times10^{-4}$ in $100\%$ of the 200 repetitions of the algorithm when $n=15$ (with about $80\%$ when $n=27$). Thus, for the considered problem of optimal DVM, SoFA with $\epsilon(k)=k^{-\alpha(k)}$ provides a far superior approximation than the ESCH, CRS and MLSL methods (orange, yellow and purple curves respectively) in terms of the speed and accuracy of its approximation. 

\begin{figure}[H]
    \centering
    \includegraphics[trim=0 0 0 0,clip,width=\textwidth]{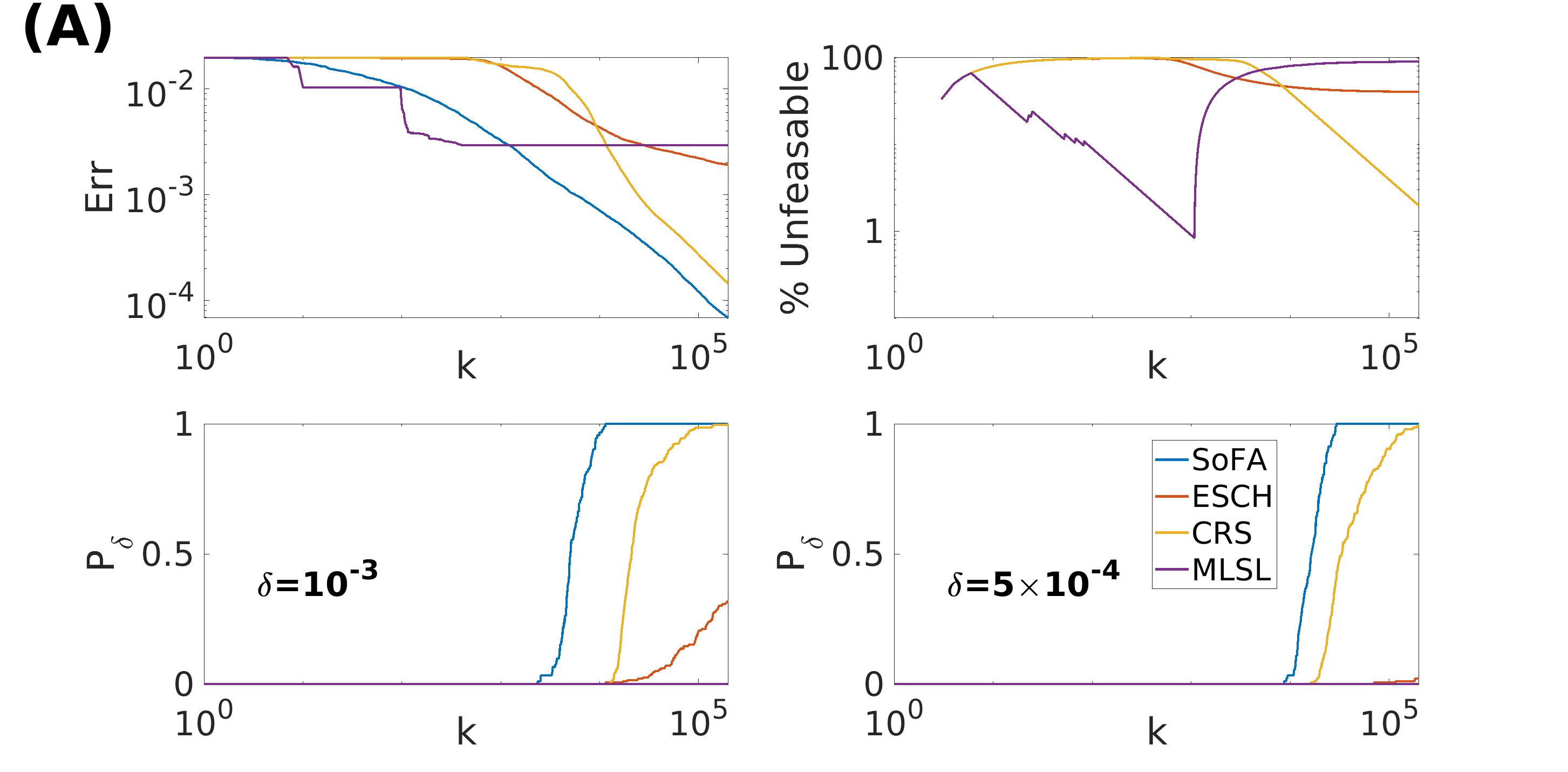}
    \includegraphics[trim=0 0 0 0,clip,width=\textwidth]{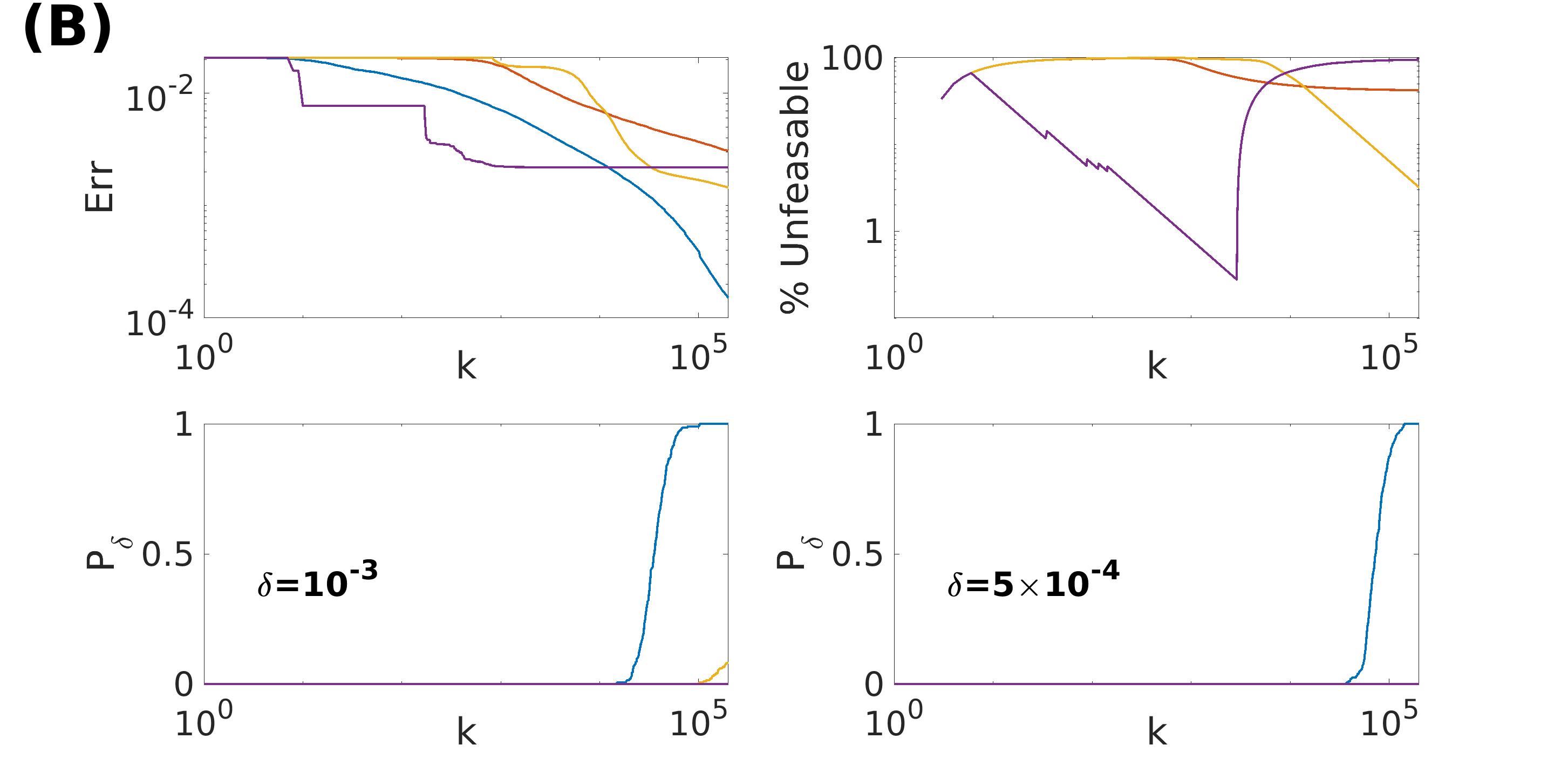}
    \caption{Implementation of multiple global optimisation algorithms to maximise fitness $J(v)$, with each trajectory $v_s$ being of the form (\ref{eq:fouier}).  (A) represent the results for an optimisation problem with a dimension of $D=45$ (i.e. $n=15$ Fourier terms per trajectory), and (B) are that but with dimension $D=81$ (i.e. $n=27$ Fourier terms per trajectory). For each block (A) and (B)
    the upper left panel shows the function error of the approximation for each iteration defined by (\ref{eq:Err}). The upper right panel shows the probability of each iteration giving an unfeasible approximation to the global maximum. The bottom panels show the probabilities of convergence to $\delta=10^{-3}$ and $\delta=5$x$10^{-4}$, as defined by (\ref{eq:P_delta}). The algorithms implemented are: SoFA with $\epsilon(k)=k^{-(a+bk)}$ with $a=0.7$ and $b=2.5$x$10^{-6}$ (blue curve) along with ESCH, CRS and MLSL (orange, yellow and purple curves respectively). The presented results are the average of 200 repetitions of each algorithm.}
    \label{fig:Compare_15}
\end{figure}
Interestingly, SoFA always provides a feasible iteration, where both the ESCH and CRS methods take some time to begin to produce some feasible points and consequently a valid approximation to the maximum. We tested the influence of increasing complexity of the system on the performance of the new optimisation method further by increasing the number of the Fourier expansions from $n=15$ to $n=27$ (see Fig.\ref{fig:Compare_15} (B)). The drastic increase in the dimensionality of the optimisation problem, $D=45$ to $D=81$, results in an overall reduction of accuracy for the fixed number of iterations (i.e. one needs more iterations to achieve the same accuracy), however, SoFA shows better performance as compared to the other optimisation methods used in this study. We found that this trend is observed for a further increase of dimensionality of the optimisation space (we do not show the corresponding figures for brevity). 

\section{Discussion}

Currently, methods of global optimisation are used intensively in various areas of research in applied mathematics, physics, economics, finance and biology. Despite the abundance of the promising techniques, novel optimisation frameworks are still being developed, in particular concerning bio-inspired algorithms. The urgent need for new methods is necessary as existing methods possess some crucial drawbacks and new practical applications require more specific problem-oriented algorithms, for example, to optimise complex biological systems in infinite-dimensional spaces. The need for developing optimisation methods efficiently working in function spaces is justified by the fact more traditional semi-analytical methods  deriving Euler-Lagrange equations, Bellman equations or Pontryagin's principle of maximum with their further numerical solution are not efficient in the case where the objective functional is non-linear, or given by a transcendental equation as in (\ref{eq:fitness}).

In this study, we introduce a bio-inspired global optimisation method named SoFA, which uses the famous idea of the survival of the fittest in biological evolution uncovered by Charles Darwin \cite{darwin2009origin}. Note that, the underlying concept of SoFA was inspired by some earlier works \cite{Grishagin}.  An important advantage of the proposed method is that one can rigorously guarantee its convergence for an extensive class of objective functions. Note that, for some bio-inspired techniques the existence of convergence has been an issue (in terms of formal proof and some simply not guaranteeing convergence with a probability of unity). Another significant advantage of SoFA is that this framework can efficiently work in higher dimensional spaces and allows a gradual increase in dimensionality to deal with infinite-dimension spaces. Our tests based on a complex biological system (\ref{eq:fitness}) showed high efficiency of SoFA compared to several other bio-inspired optimisation methods (see Fig.\ref{fig:Compare_15}). Finally, the algorithm itself is simple in terms of practical coding. Moreover, the given algorithm allows us to use parallel programming, for example one can implement a modification of SoFA by introducing new dimensions by entire blocks and fulfil optimisation for entire blocks.   

As it follows from its nature, the algorithm can be applied to uncover the optimal fitness in complex biological systems with multiple maxima, as the considered example of plankton DVM. However, the method can be naturally implemented to other non-biological systems, for example, in problems of optimal control of heat transfer \cite{Grishagin} or deformations in a metal rod \cite{irkhina2005identification}. Mathematically, the convergence of our method is guaranteed by the property of the localisation of the probability measure in a Hilbert space, which, was reported in previous studies of modelling biological evolution \cite{gorban2007selection,Kuzenkov_measure,Kuzenkov_N,kuzenkov2019towards}. In this case, the measure of available strategies (life-history traits) within the population should be the indicator of the presence of a set of strategies in this population. This measure (in the simplest case it is the amount organisms using a particular strategy) will evolve with time, i.e. with an increase of the number of iterations. As a result, after a long time, only strategies with the fitness close to the optimal one would survive in the system. Eventually, the distribution (e.g. the density function) of strategies would tend to a delta function with a centre corresponding to the maximum of fitness. We call this phenomenon the localisation of the measure.

The phenomenon of the localisation of measure in SoFA can be seen better in Fig. \ref{fig:Distributions}, which shows the evolution of the distribution of the probability density function for the point $z_{k+1}$ after $k$ first iterations shown in the figure label. For simplicity, we show the evolution of the constant term $v_{Y,1}$ in the Fourier expansion for the stage Y.  The resulting distributions in Fig. \ref{fig:Distributions} are obtained using the total probability law theorem: the conditional probability to find $z_{k+1}$  for a particular choice of the reference point $\bar{z}$ is averaged across all possible reference points with weights given by $\frac{J^k(z_i)}{J^k(z_1)+\ldots +J^k(z_k)}$, $i=1,\ldots ,k$. One can see from the figure that the distribution of the population of strategies is drifting on a logarithmic time scale towards a certain final value of parameter value and the shape of the distribution is approaching a delta function. As a result, the measure (the integral over the density) is eventually being localised at the point corresponding to the maximum of fitness. Note that the closeness of the distribution to the delta function can be a working criterion of terminating interactions in the considered algorithm. Interestingly, plotting the distributions of iteration points as those in  Fig. \ref{fig:Distributions} can be useful for assessing the standard deviation of life-history traits within a population as a function of the number of population generations. This will of practical biological (i.e. not only mathematical) interest in the case a biologist needs to estimate the scatting of the considered life-history trait of evolving organisms from the expected eventual optimal value.


\begin{figure}[H]
    \centering
    \includegraphics[trim=140 0 200 50,clip,width=\textwidth]{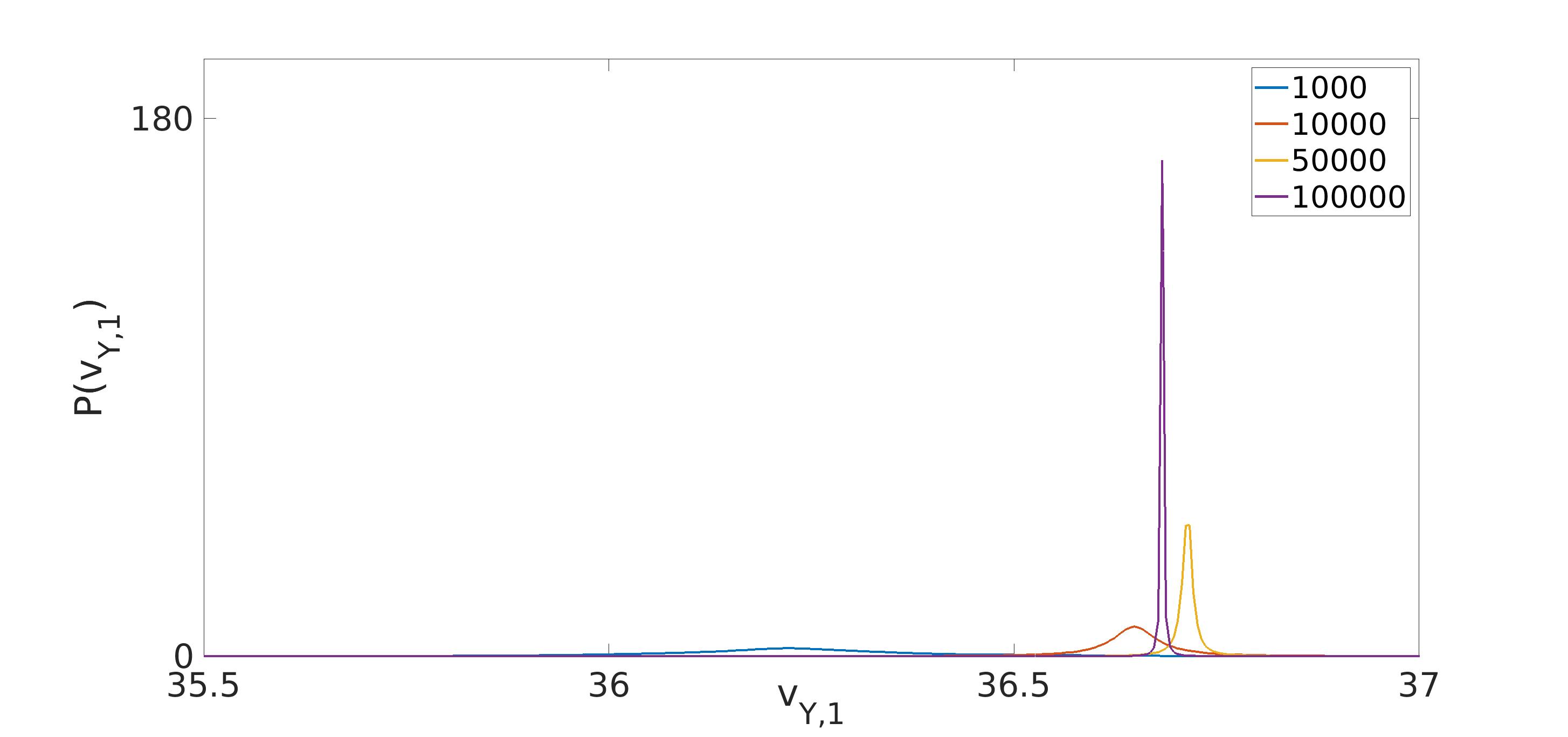}
    \caption{Probability density distribution of the parameter $v_{Y,1}$ (the average daily depth in the DVM trajectory of early developmental stage Y). Here the new optimisation method is implemented with $\epsilon(k)=k^{-(a+bk)}$ with $a=0.7$ and $b=5$x$10^{-6}$ to the maximisation of the fitness $J(v)$, with each trajectory $v_s$ being of the form (\ref{eq:fouier})  each with $n=15$ terms.}
    \label{fig:Distributions}
\end{figure}

Interestingly, our implementation of the SoFA framework to reveal the DVM of herbivorous zooplankton provides some important ecological insights. Note that the model parameters describing the growth, morality, maturation, energy losses, etc. are biologically meaningful and correspond to those in \cite{morozov2019}. For these realistic parameters, we reveal the existence of another (suboptimal) scenario of DVM where fitness has its local maximum. This scenario is characterised by zooplankton staying in shallower depths during the day time  (see Fig.\ref{fig:Low_Order}). Biologically, the shallow depths scenario signifies that zooplankton grazers reach the minimal depth, where the visual predation by fish becomes negligible. Whereas, empirical observations confirm that, in reality, the other scenario where zooplankton stay during the day time at much deeper depths is realised \cite{morozov2019}. The deep waters DVM scenario corresponds to the global maximum of population fitness $J$ shown in Fig.\ref{fig:Low_Order}. The relevant biological conclusion is that the optimal migration depth in the considered ecosystem (the north-eastern Black Sea), is determined by metabolic costs (which are negligible in deep waters) rather than the predation threat which becomes negligible already at depths of 40-50m. Another interesting conclusion from the implementation of SoFA is that the smooth nonlinear optimal trajectories of DVM modelled in the entire function space are close to those given by linear functions at the ascending, descending and deep waters phase of DVM. On the contrary, near the water surface (during the night time), a constant adjustment of depth by a zooplankter results in a higher value of fitness, compared to the scenario where the organism remains at a constant depth.

Note that as a stochastic optimisation method, SoFA is close to the evolutionary algorithms, such as differential evolution \cite{storn1997differential}.  However, an important difference of SoFA with the existing evolutionary algorithms is that it uses the information about the trial points obtained from the all previous iterations (i.e. using the entire evolution history) when choosing the reference point at the current iteration step. As such, the selection of strategies in SoFA occurs on a slower time scale. Also, unlike genetic algorithms and differential evolution, which are semi-empirical, SoFA assures the convergence for an arbitrary positive continuous functional $J$.

The proposed method is an extension of a Monte-Carlo classical approach, which, uses a non-uniform and adjustable density of distribution of trial points.

Following from our numerical experiments, the convergence rate of SoFA largely depends on how quickly the value of the standard deviation in the distribution of $z_k$ decreases with the number of iterations. A fast drop in the mentioned standard deviation results in an improvement of the convergence rate; however, this also increases the risk of being stuck for a long time near a suboptimal local maximum. In the case the objective function does not contain abrupt spikes, the risk of being trapped by a non-global maximum is low. For a sufficiently smooth function $J$, one can substantially accelerate the convergence rate via increasing the rate of decay of the standard deviation of the distribution of $z_k$. Therefore, the method has a potential of improving the convergence rate via an \textit{a priori} knowledge about the behaviour of the objective function. The required \textit{a priori} knowledge of $J$ can be obtained as a part of the implementation of the algorithm, for example by estimating the upper bound of the Lipschitz constant for different parts of the parameter space. Using the evaluation of $J$ at each trial point, one can estimate the local Lipschitz constant. Using the obtained estimates of local Lipschitz constants, one can group the trial points $z_k$ by splitting the whole parameter space into subdomains with different patterns of behaviour of the functions: `flat valleys' and `sharp peaks'. For different subdomains, one can consider distinct rates of increase of the standard deviations, which will allow for the enhancing of the convergence properties without risk to get tapped by a non-global maximum.

The intrinsic connection between SoFA and biological evolution can be uncovered when considering the growth of the dimensionality of the search space, which biologically signifies a gradual increase in the complexity of competing species, macroevolution. The method generates mutations not only by modifying the values of the existing model parameters but also producing new types of mutants with characteristics absent in ancestors via adding new dimensions. This is related to the increase of the number of inherited life-history traits and the overall length of the underlying genetic code. Since the increase in dimensionality in the method is not bounded, the evolutionary processes modelled by SoFA admit an unlimited enhancement of quality species and an unrestricted increase in the complexity of their organisation.

Finally, the presented optimisation method would contribute to the resolving of the curse of dimensionality problem in optimisation. For a hyperball in $n$ dimensions, its volume is concentrated near its boundary. Signifying that in the case where we generate trial points based on a uniform distribution inside some $n$-dimensional hyperball, most of the points will lend within a very thin boundary layer of this hyperball, whereas, only a small proportion of them will target the central part \cite{gorban2018blessing}. However, from the classical theory of calculus of variations, it is well-known that the solution to the optimisation problem is usually an internal point of the functional space. Therefore, a numerical solution of the optimisation problem in high dimensional spaces approximating the underlying Hilbert space requires approaching to some internal point of the considered multi-dimensional hyperball. All stochastic methods based on the uniform distribution of trial points become largely inefficient in this situation, whereas SoFA would be a better candidate to cope with the challenge. Indeed, with a further increase of dimensionality of the search space, the distribution of generated points in SoFA experiences a permanent evolution of the shape. This guarantees the localisation of the trial points in internal parts (i.e. located far from the boundary) of the multi-dimensional hyperballs (see Fig. \ref{fig:Distributions} as an illustrative example). 



\section{Summary}

In this paper, we present a bio-inspired method of global optimisation which quantifies Darwin's' famous idea of the survival of the fittest (the Survival of the Fittest Algorithm, SoFA). The method has multiple advantages as compared to other bio-inspired stochastic optimisation algorithms. In particular, its convergence is guaranteed for any positive continuous objective function(al) and one can apply the method can cope with increasing dimensionality of space: it can find the optimal solution in an infinite-dimensional functional space. Based on an insightful motivating example, maximisation of the fitness functional in a stage-dependent population model, we demonstrate the better performance of SoFA, as compared with some other stochastic algorithms of global optimisation in the case when the dimensionality of the parameter space is high. 

\section*{Appendix A. Simplified algorithm of the optimisation method.}
Here, we provided a simplified version of the SoFA framework presented in Section \ref{sec:Algorithm}.

Suppose there is some continuous positive function $J(z)$ which is defined on the rectangular domain $P=\{z=[z^1,\ldots,z^n]:a^i<z^i<b^i,i=1:n\}$. Assuming that $J(z)$ has a unique point of maximum (denoted by $z^*$), such a point can be located by implementing a simplification of the novel optimisation method described in section 1. This simplified method is outlined in the following algorithm.
\begin{enumerate}
  \item Assume $k$ steps of the method have already been completed and therefore one has $k$ points in $P=\{z_1,..,z_k\}$, with known corresponding values of $J(z_1),\ldots,J(z_k)$.
  \item Randomly select some reference point $\bar{z}$ out of all the available points in $P$. The probability of selecting each of these available points is given by $\frac{J^k(z_i)}{J^k(z_1)+\ldots+J^k(z_k)}$
  \item Using this reference point, select a new point $z_{k+1}$ by setting each coordinate of this point to be a random variable in the interval $[a^j,b^j]$ with the probability density function given by $\frac{A^j_{k+1}}{\epsilon_{k+1}+((z_{k+1}^j)-\bar{z}^j)^2}$ where the constants $A^j_{K+1}$ are chosen to normalise the probability over the interval $[a^j,b^j]$. Here $\bar{z}^j$ is $j^{th}$ coordinate the reference point $\bar{z}$. Note that we use the same notation for the reference point $\bar{z}$ for the sake of simplicity: at each step the reference point might be different.
  \item The function is evaluated at $z_{k+1}$ giving the value of $J(z_{k+1})$, then the steps can be repeated.
  \item The method will terminate once the approximation satisfies some initially prescribed requirement (for example, the distribution of $z_{k+1}$ is sufficiently close to a delta function, see Fig.4).
\end{enumerate}

\bibliography{references}

\begin{thebibliography}{10}
\expandafter\ifx\csname url\endcsname\relax
  \def\url#1{\texttt{#1}}\fi
\expandafter\ifx\csname urlprefix\endcsname\relax\def\urlprefix{URL }\fi
\expandafter\ifx\csname href\endcsname\relax
  \def\href#1#2{#2} \def\path#1{#1}\fi

\bibitem{deb2001multi}
K.~Deb, Multi-objective optimization using evolutionary algorithms, Vol.~16,
  John Wiley \& Sons, 2001.

\bibitem{passino2012bacterial}
K.~M. Passino, Bacterial foraging optimization, in: Innovations and
  Developments of Swarm Intelligence Applications, IGI Global, 2012, pp.
  219--234.

\bibitem{mavrovouniotis2017survey}
M.~Mavrovouniotis, C.~Li, S.~Yang, A survey of swarm intelligence for dynamic
  optimization: Algorithms and applications, Swarm and Evolutionary Computation
  33 (2017) 1--17.

\bibitem{rai2013bio}
D.~Rai, K.~Tyagi, Bio-inspired optimization techniques: a critical comparative
  study, ACM SIGSOFT Software Engineering Notes 38~(4) (2013) 1--7.

\bibitem{back1996evolutionary}
T.~Back, Evolutionary algorithms in theory and practice: evolution strategies,
  evolutionary programming, genetic algorithms, Oxford university press, 1996.

\bibitem{storn1997differential}
R.~Storn, K.~Price, Differential evolution--a simple and efficient heuristic
  for global optimization over continuous spaces, Journal of global
  optimization 11~(4) (1997) 341--359.

\bibitem{di2015optimal}
F.~Di~Patti, D.~Fanelli, F.~Piazza, Optimal search strategies on complex
  multi-linked networks, Scientific reports 5~(1) (2015) 1--6.

\bibitem{volchenkov2013exploration}
D.~Volchenkov, J.~Helbach, M.~Tscherepanow, S.~K{\"u}hnel,
  Exploration--exploitation trade-off features a saltatory search behaviour,
  Journal of The Royal Society Interface 10~(85) (2013) 20130352.

\bibitem{odum1971fundamentals}
E.~P. Odum, G.~W. Barrett, Fundamentals of ecology, Vol.~3, Saunders
  Philadelphia, 1971.

\bibitem{darwin2009origin}
C.~Darwin, W.~F. Bynum, The origin of species by means of natural selection:
  or, the preservation of favored races in the struggle for life, AL Burt New
  York, 2009.

\bibitem{gorban2007selection}
A.~N. Gorban, Selection theorem for systems with inheritance, Mathematical
  Modelling of Natural Phenomena 2~(4) (2007) 1--45.

\bibitem{Kuzenkov_measure}
O.~Kuzenkov, E.~Ryabova, Limit possibilities of solution of a hereditary
  control system, Differential Equations 51~(4) (2015) 523--532.

\bibitem{Kuzenkov_N}
O.~Kuzenkov, A.~Novozhenin, Optimal control of measure dynamics, Communications
  in Nonlinear Science and Numerical Simulation 21~(1) (2015) 159 -- 171.

\bibitem{kuzenkov2019towards}
O.~Kuzenkov, A.~Morozov, Towards the construction of a mathematically rigorous
  framework for the modelling of evolutionary fitness, Bulletin of Mathematical
  Biology 81~(11) (2019) 4675--4700.

\bibitem{hays2003}
G.~C. Hays, A review of the adaptive significance and ecosystem consequences of
  zooplankton diel vertical migrations, in: Migrations and Dispersal of Marine
  Organisms, Springer, 2003, pp. 163--170.

\bibitem{kaiser2011}
M.~J. Kaiser, M.~J. Attrill, S.~Jennings, D.~N. Thomas, D.~K. Barnes, et~al.,
  Marine ecology: processes, systems, and impacts, Oxford University Press,
  2011.

\bibitem{da2010}
C.~H. da~Silva~Santos, M.~S. Goncalves, H.~E. Hernandez-Figueroa, Designing
  novel photonic devices by bio-inspired computing, IEEE Photonics Technology
  Letters 22~(15) (2010) 1177--1179.

\bibitem{da2010Thesis}
C.~H. da~Silva~Santos, Parallel and bio-inspired computing applied to analyze
  microwave and photonic metamaterial strucutures., Ph.D. thesis, University of
  Campinas, Brazil (2010).

\bibitem{kaelo2006}
P.~Kaelo, M.~Ali, Some variants of the controlled random search algorithm for
  global optimization, Journal of optimization theory and applications 130~(2)
  (2006) 253--264.

\bibitem{price1983}
W.~Price, Global optimization by controlled random search, Journal of
  Optimization Theory and Applications 40~(3) (1983) 333--348.

\bibitem{price1977}
W.~L. Price, A controlled random search procedure for global optimisation, The
  Computer Journal 20~(4) (1977) 367--370.

\bibitem{kan1987I}
A.~R. Kan, G.~T. Timmer, Stochastic global optimization methods part i:
  Clustering methods, Mathematical programming 39~(1) (1987) 27--56.

\bibitem{kan1987II}
A.~R. Kan, G.~T. Timmer, Stochastic global optimization methods part ii: Multi
  level methods, Mathematical Programming 39~(1) (1987) 57--78.

\bibitem{ducklow2001}
H.~W. Ducklow, D.~K. Steinberg, K.~O. Buesseler, Upper ocean carbon export and
  the biological pump, OCEANOGRAPHY-WASHINGTON DC-OCEANOGRAPHY SOCIETY- 14~(4)
  (2001) 50--58.

\bibitem{buesseler2007}
K.~O. Buesseler, C.~H. Lamborg, P.~W. Boyd, P.~J. Lam, T.~W. Trull, R.~R.
  Bidigare, J.~K. Bishop, K.~L. Casciotti, F.~Dehairs, M.~Elskens, et~al.,
  Revisiting carbon flux through the ocean's twilight zone, Science 316~(5824)
  (2007) 567--570.

\bibitem{bianchi2013}
D.~Bianchi, E.~D. Galbraith, D.~A. Carozza, K.~Mislan, C.~A. Stock,
  Intensification of open-ocean oxygen depletion by vertically migrating
  animals, Nature Geoscience 6~(7) (2013) 545--548.

\bibitem{hansen2016}
A.~N. Hansen, A.~W. Visser, Carbon export by vertically migrating zooplankton:
  an optimal behavior model, Limnology and Oceanography 61~(2) (2016) 701--710.

\bibitem{ohman1994}
M.~D. Ohman, J.~A. Runge, Sustained fecundity when phytoplankton resources are
  in short supply: omnivory by calanus finmarchicus in the gulf of st.
  lawrence, Limnology and Oceanography 39~(1) (1994) 21--36.

\bibitem{fortier2001}
M.~Fortier, L.~Fortier, H.~Hattori, H.~Saito, L.~Legendre, Visual predators and
  the diel vertical migration of copepods under arctic sea ice during the
  midnight sun, Journal of Plankton Research 23~(11) (2001) 1263--1278.

\bibitem{pearre2003}
S.~Pearre, Jr, Eat and run? the hunger/satiation hypothesis in vertical
  migration: history, evidence and consequences, Biological Reviews 78~(1)
  (2003) 1--79.

\bibitem{lampert1989}
W.~Lampert, The adaptive significance of diel vertical migration of
  zooplankton, Functional ecology 3~(1) (1989) 21--27.

\bibitem{clark2000}
C.~W. Clark, M.~Mangel, et~al., Dynamic state variable models in ecology:
  methods and applications, Oxford University Press on Demand, 2000.

\bibitem{ringelberg2009}
J.~Ringelberg, Diel vertical migration of zooplankton in lakes and oceans:
  causal explanations and adaptive significances, Springer Science \& Business
  Media, 2009.

\bibitem{morozov2011}
A.~Morozov, E.~Arashkevich, A.~Nikishina, K.~Solovyev, Nutrient-rich plankton
  communities stabilized via predator—prey interactions: revisiting the role
  of vertical heterogeneity, Mathematical Medicine and Biology: a Journal of
  the IMA 28~(2) (2011) 185--215.

\bibitem{Fiksen95}
{\O}.~Fiksen, J.~Giske, Vertical distribution and population dynamics of
  copepods by dynamic optimization, ICES Journal of Marine Science 52~(3-4)
  (1995) 483--503.
\newblock \href {https://doi.org/10.1016/1054-3139(95)80062-X}
  {\path{doi:10.1016/1054-3139(95)80062-X}}.

\bibitem{han2001control}
B.-P. Han, M.~Stra{\v{s}}kraba, Control mechanisms of diel vertical migration:
  theoretical assumptions, Journal of Theoretical Biology 210~(3) (2001)
  305--318.

\bibitem{liu2003diel}
S.-H. Liu, S.~Sun, B.-P. Han, Diel vertical migration of zooplankton following
  optimal food intake under predation, Journal of Plankton Research 25~(9)
  (2003) 1069--1077.

\bibitem{morozov2016towards}
A.~Y. Morozov, O.~A. Kuzenkov, Towards developing a general framework for
  modelling vertical migration in zooplankton, Journal of Theoretical Biology
  405 (2016) 17--28.

\bibitem{morozov2019}
A.~Morozov, O.~A. Kuzenkov, E.~G. Arashkevich, Modelling optimal behavioural
  strategies in structured populations using a novel theoretical framework,
  Scientific reports 9~(1) (2019) 1--15.

\bibitem{cushing1998}
J.~M. Cushing, An introduction to structured population dynamics, SIAM, 1998.

\bibitem{botsford1994}
L.~W. Botsford, B.~D. Smith, J.~F. Quinn, Bimodality in size distributions: the
  red sea urchin strongylocentrotus franciscanus as an example, Ecological
  Applications 4~(1) (1994) 42--50.

\bibitem{Grishagin}
O.~A. Kuzenkov, V.~A. Grishagin, Global optimization in hilbert space, AIP
  Conference Proceedings 1738~(1) (2016) 400007.
\newblock \href {https://doi.org/10.1063/1.4952195}
  {\path{doi:10.1063/1.4952195}}.

\bibitem{irkhina2005identification}
A.~Irkhina, O.~Kuzenkov, Identification of the distribution of deformations in
  a rod as a problem of optimal control, Journal of Computer and Systems
  Sciences International 44~(5) (2005) 689--694.

\bibitem{gorban2018blessing}
A.~N. Gorban, I.~Y. Tyukin, Blessing of dimensionality: mathematical
  foundations of the statistical physics of data, Philosophical Transactions of
  the Royal Society A: Mathematical, Physical and Engineering Sciences
  376~(2118) (2018) 20170237.

\end{thebibliography}

\end{document}